\documentclass[12pt,a4paper]{article}

\usepackage{amssymb,amsmath,amsthm}
\usepackage[english]{babel}
\usepackage{t1enc}
\usepackage[latin2]{inputenc}
\usepackage{epsfig}
\usepackage{epstopdf}

\newtheorem{thm}{Theorem}
\newtheorem{cor}[thm]{Corollary}
\newtheorem{defi}[thm]{Definition}
\newtheorem{lem}[thm]{Lemma}

\newtheorem{claim}[thm]{Claim}

\newtheorem{obs}[thm]{Observation}

\newtheorem{conj}{Conjecture}

\usepackage{indentfirst}
\newcommand{\ta}{t}
\newcommand{\te}{t_e}

\title{On the number of edge-disjoint triangles in $K_4$-free graphs}

\author{Ervin Gy\H{o}ri\thanks{Research partially supported by Hungarian National Science Fund (OTKA) grant 101536.} \\
\small R\'enyi Institute\\[-0.8ex]
\small Hungarian Academy of Sciences\\[-0.8ex]
\small Budapest, Hungary\\
\small\tt gyori.ervin@renyi.mta.hu\\
\and
Bal\'azs Keszegh \thanks{Research supported by Hungarian National Science Fund (OTKA) grant PD 108406 and by the J\'anos Bolyai Research Scholarship of the Hungarian Academy of Sciences.
}\\
\small R\'enyi Institute\\[-0.8ex]
\small Hungarian Academy of Sciences\\[-0.8ex]
\small Budapest, Hungary\\
\small\tt keszegh.balazs@renyi.mta.hu\\
}
\date{}

\begin{document}

\maketitle

\begin{abstract}
We show the quarter of a century old conjecture that every $K_4$-free graph with $n$ vertices and $\lfloor n^2/4 \rfloor +k$ edges contains $k$ pairwise edge disjoint triangles.
\end{abstract}

\section{Introduction}
Extending the well-known result of extremal graph theory by Tur\'an, E. Gy\H ori and A.V. Kostochka \cite{ek} and independently F.R.K Chung \cite{chung} proved the following theorem. For an arbitrary graph $G$, let $p(G)$ denote the minimum of $\sum|V(G_i)|$ over all decompositions of $G$ into edge disjoint cliques $G_1,G_2,...$. Then $p(G)\le2t_2(n)$ and equality holds if and only if $G\cong T_2(n)$. Here $T_2(n)$ is the $2$-partite Tur\'an graph on $n$ vertices and $t_2(n)=\lfloor n^2/4\rfloor$ is the number of edges of this graph. P. Erd\H os later suggested to study the weight function $p^*(G)=\min \sum(|V(G_i)|-1))$. The first author   \cite{ervinsurvey} started to study this function and to prove the conjecture $p^*(G)\le t_2(n)$ just in the special case when $G$ is $K_4$-free.  This 24 year old conjecture was worded equivalently as follows.
\begin{conj} \label{mainconj}
Every $K_4$-free graph on $n$ vertices and $t_2(n)+m$ edges contains at least $m$ edge disjoint triangles.
\end{conj}
This was only known if the graph is $3$-colorable i.e. $3$-partite.

In \cite{chinese} towards proving the conjecture, they proved that for every $K_4$-free graph there are always at least $32k/35\ge 0.9142k$ edge-disjoint triangles and if $k\ge 0.0766 n^2$ then there are at least $k$ edge-disjoint triangles.
Their main tool is a nice and simple to prove lemma connecting the number of edge-disjoint triangles with the number of all triangles in a graph. In this paper using this lemma and proving new bounds about the number of all triangles in $G$, we settle the above conjecture:

\begin{thm}\label{thm:main}
Every $K_4$-free graph on $n^2/4+k$ edges contains at least $\lceil k\rceil$ edge-disjoint triangles.
\end{thm}

This result is best possible, as there is equality in Theorem \ref{thm:main} for every graph which we get by taking a $2$-partite Tur\'an graph and putting a triangle-free graph into one side of this complete bipartite graph. Note that this construction has roughly at most $n^2/4+n^2/16$ edges  while in general in a $K_4$-free graph $k\le n^2/12$, and so it is possible (and we conjecture so) that an even stronger theorem can be proved if we have more edges, for further details see section Remarks.

\section{Proof of Theorem \ref{thm:main}}

From now on we are given a graph $G$ on $n$ vertices and having $e=n^2/4+k$ edges.

\begin{defi}
Denote by $\te$ the maximum number of edge disjoint triangles in $G$ and  by $\ta$ the number of all triangles of $G$.
\end{defi}

The idea is to bound $\te$  by $\ta$. For that we need to know more about the structure of $G$, the next definitions are aiming towards that.

\begin{defi}
A {\bf good partition} $P$ of $V(G)$ is a partition of $V(G)$ to disjoint sets $C_i$ (the cliques of $P$) such that every $C_i$ induces a complete subgraph in $G$.

The {\bf size} $r(P)$ of a good partition $P$ is the number of cliques in it. The cliques of a good partition $P$ are ordered such that their size is non-decreasing: $|C_0|\le|C_1|\le\dots \le|C_{r(P)}|$.

A good partition is a {\bf greedy partition} if for every $l\ge 1$ the union of all the parts of size at most $l$ induces a $K_{l+1}$-free subgraph, that is, for every $i\ge 1$, $C_0\cup C_1\cup\dots\cup C_i$ is $K_{|C_i|+1}$-free.
(See Figure \ref{fig:defgp} for examples.)
\end{defi}

{\it Remark.} In our paper $l$ is at most 3 typically, but in some cases it can be arbitrary.

Note that the last requirement in the definition holds also trivially for $i=0$.

The name greedy comes from the fact that a good partition is a greedy partition if and only if we can build it greedily in backwards order, by taking a maximal size complete subgraph $C\subset V(G)$ of $G$ as the last clique in the partition, and then recursively continuing this process on $V(G)\setminus C$ until we get a complete partition. This also implies that every $G$ has at least one greedy partition. If $G$ is $K_4$-free then a greedy partition is a partition of $V(G)$ to $1$ vertex sets, $2$ vertex sets spanning an edge and $3$ vertex sets spanning a triangle, such that the union of the size 1 cliques of $P$ is an independent set and the union of the size 1 and size 2 cliques of $P$ is triangle-free.

\begin{figure}[t]
    \centering
        \includegraphics[scale=1]{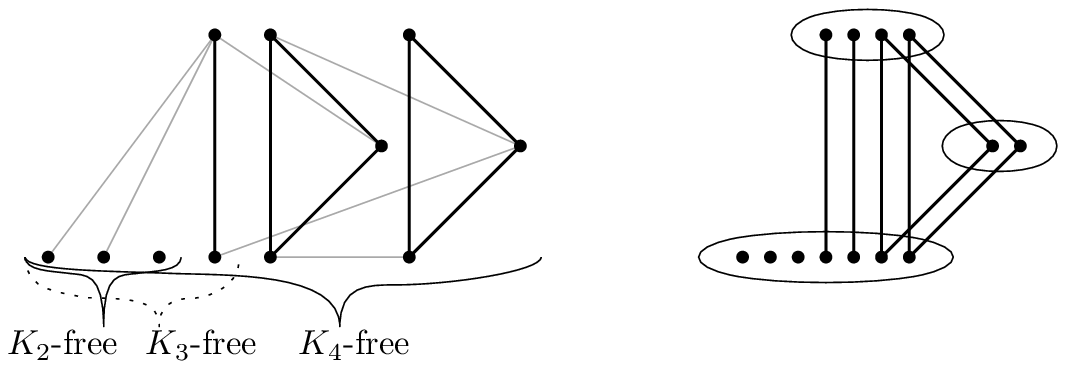}
   \caption{A greedy partition of an arbitrary graph and of a complete $3$-partite graph.}
   \label{fig:defgp}
\end{figure}



\begin{lem}[\cite{chinese}]
Let $G$ be a $K_4$-free graph and $P$ be a greedy partition of $G$. Then
$$\te\ge \frac{\ta}{r(P)}.$$
\end{lem}

For sake of keeping the paper self-contained, we prove this lemma too.

\begin{proof}
Let $r=r(P)$ and the cliques of the greedy partition be $C_0,C_1,\dots C_{r-1}$. With every vertex $v\in C_i$ we associate the value $h(v)=i$ and with every triangle of $G$ we associate the value $h(T)=\sum _{v\in T} h(v) \mod r$. As there are $r$ possible associated values, by the pigeonhole principle there is a family $\cal T$ of at least $\ta/r$ triangles that have the same associated value. It's easy to check that two triangles sharing an edge cannot have the same associated value if $G$ is $K_4$-free, thus $\cal T$ is a family of at least $\ta/r$ edge-disjoint triangles in $G$, as required.
\end{proof}

 It implies that $\te\ge \frac{\ta}{R(P)}$, moreover the inequality is true for every $P$.
Note that the next theorem holds for every graph, not only for $K_4$-free graphs.

\begin{thm}\label{thm:tbound}
Let $G$ be a graph and $P$ a greedy partition of $G$. Then $t\ge r(P)\cdot(e-n^2/4)$.
\end{thm}

By choosing an arbitrary greedy partition $P$ of $G$, the above lemma and theorem together imply that for a $K_4$-free $G$ we have $t_e\ge \frac{\ta}{r(P)}\ge e-n^2/4=k$, concluding the proof of Theorem \ref{thm:main}.

\bigskip

Before we prove Theorem \ref{thm:tbound}, we make some preparations.

\begin{lem}\label{lem:twocliques}
Given a $K_{b+1}$-free graph $G$ on vertex set $A\cup B$, $|A|=a\le b=|B|$, $A$ and $B$ both inducing complete graphs, there exists a matching of non-edges between $A$ and $B$ covering $A$. In particular, $G$ has at least $a$ non-edges.
\end{lem}
\begin{proof}
Denote by $\bar G$ the complement of $G$ (the edges of $\bar G$ are the non-edges of $G$).
To be able to apply Hall's theorem, we need that for every subset $A'\subset A$ the neighborhood $N(A')$ of $A'$ in $\bar G$ intersects $B$ in at least $|A'|$ vertices. Suppose there is an $A'\subset A$ for which this does not hold, thus for $B'=B\setminus N(A')$ we have $|B'|= |B|-|B\cap N(A')|\ge b-(a-1)$. Then $A'\cup B'$ is a complete subgraph of $G$ on at least $a+b-(a-1)=b+1$ vertices, contradicting that $G$ is $K_{b+1}$-free.
\end{proof}

\begin{obs}\label{obs:Pdoesnotmatter}
If $G$ is complete $l$-partite for some $l$ then it has essentially one greedy partition, i.e., all greedy partitions of $G$ have the same clique sizes and have the same number of cliques, which is the size of the biggest part (biggest independent set) of $G$.
\end{obs}

We regard the following function depending on $G$ and $P$ (we write $r=r(P))$: $$f(G,P)=r(e-n^2/4)-t.$$ We are also interested in the function $$g(G,P)=r(e-r(n-r))-t.$$ Notice that $g(G,P)\ge f(G,P)$ and  $f$ is a monotone increasing function of $r$ (but $g$ is not!) provided that $e-n^2/4\ge 0$. Also, using Observation \ref{obs:Pdoesnotmatter} we see that if $G$ is complete multipartite then $r$, $f$ and $g$ do not depend on $P$, thus in this case, we may write simply $f(G)$ and $g(G)$.

\begin{lem}\label{lem:completepartite}
If $G$ is a complete $l$-partite graph then $g(G)\le 0$ and if $G$ is complete $3$-partite (some parts can have size $0$) then $g(G)= 0$.
\end{lem}
\begin{proof}
Let $G$ be a complete $l$-partite graph with part sizes $c_1\le \dots \le c_l$. By Observation \ref{obs:Pdoesnotmatter}, $r=r(P)=c_l$ for any greedy partition.
We have $n=\sum_i c_i$, $e=\sum_{i< j} c_ic_j$, $t=\sum_{i< j< m} c_ic_jc_m$ and so
$$g(G)= r(e-r(n-r))-t=c_l(\sum_{i< j} c_ic_j-c_l\sum_{i< l} c_i)-t=$$ $$=c_l\sum_{i<j<l} c_ic_j-\sum_{i<j<m} c_ic_jc_m=-\sum_{i<j<m<l} c_ic_jc_m\le 0.$$
Moreover, if $l\le 3$ then there are no indices $i<j<m<l$ thus the last equality also holds with equality.
\end{proof}

\begin{figure}[t]
    \centering
        \includegraphics[scale=1]{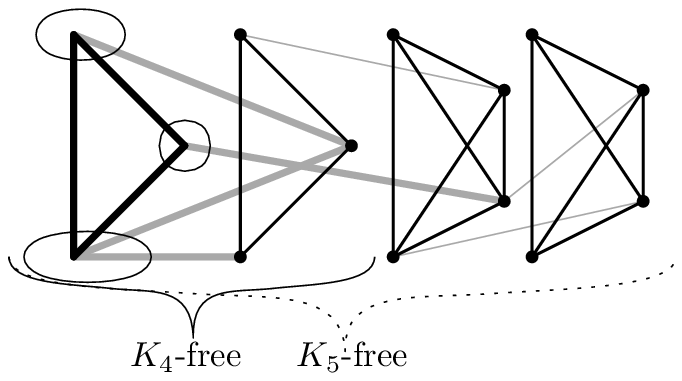}
   \caption{A generalized greedy partition of an arbitrary graph (heavy edges represent complete bipartite graphs) .}
   \label{fig:defggp}
\end{figure}

In the proof we need a generalization of a greedy partition, which is similar to a greedy partition, with the only difference that the first part $C_0$ in the partition $P$ is a blow-up of a clique instead of a clique. see Figure \ref{fig:defggp} for an example.

\begin{defi}
A $P$ {\bf generalized greedy partition} ({\bf ggp} in short) of some graph $G$ is a partition of $V(G)$ into the sequence of disjoint sets $C_0,C_1,\dots C_l$ such that $C_0$ induces a complete $l_0$-partite graph, $C_i, i\ge 1$ induces a clique and $l_0\le |C_1|\le \dots |C_l|$. We require that for every $i\ge 1$, $C_0\cup C_1\cup\dots\cup C_i$ is $K_{|C_i|+1}$-free.

We additionally require that if two vertices are not connected in $C_0$ (i.e., are in the same part of $C_0$) then they have the same neighborhood in $G$, i.e., vertices in the same part of $C_0$ are already symmetric.

The {\bf size} $r(P)$ of a greedy partition $P$ is defined as the size of the biggest independent set of $C_0$ plus $l-1$, the number of parts of $P$ besides $C_0$.
\end{defi}

Note that the last requirement in the definition holds also for $i=0$ in the natural sense that $C_0$ is $l_0+1$-free.

Observe that the requirements guarantee that in a ggp $P$ if we contract the parts of $C_0$ (which is well-defined because of the required symmetries in $C_0$) then $P$ becomes a normal (non-generalized) greedy partition (of a smaller graph).

Using Observation \ref{obs:Pdoesnotmatter} on $C_0$, we get that the size of a ggp $P$ is equal to the size of any underlying (normal) greedy partition $P'$ of $G$ which we get by taking any greedy partition of $C_0$ and then the cliques of $P\setminus\{C_0\}$. Observe that for the sizes of $P$ and $P'$ we have $r(P)=r(P')$, in fact this is the reason why the size of a ggp is defined in the above way.

Finally, as we defined the size $r(P)$ of a ggp $P$, the definitions of the functions $f(G,P)$ and $g(G,P)$ extend to a ggp $P$ as well. With this notation Lemma \ref{lem:completepartite} is equivalent to the following:

\begin{cor}\label{cor:onepartggp}
If a ggp $P$ has only one part $C_0$, which is a blow-up of an $l_0$-partite graph, then $r(G,P)\le 0$ and if $l_0\le 3$ then $r(G,P)=0$.
\end{cor}

\begin{proof}[Proof of Theorem \ref{thm:tbound}]

The theorem is equivalent to the fact that for every graph $G_0$ and greedy partiton $P_0$ we have $f(G_0,P_0)\le 0$.

Let us first give a brief summary of the proof. We will repeatedly do some symmetrization steps, getting new graphs and partitions, ensuring that during the process $f$ cannot decrease. At the end we will reach a complete $l$-partite graph $G_*$ for some $l$. However by Lemma \ref{lem:completepartite} for such graphs $g(G_*,P_*)\le 0$ independent of $P_*$, which gives $f(G_0,P_0)\le f(G_*)\le g(G_*)\le 0$. This proof method is similar to the proof from the book of Bollob\'as \cite{bollobas} (section VI. Theorem 1.7.) for a (not optimal) lower bound on $t$ by a function of $e,n$. An additional difficulty comes from the fact that our function also depends on $r$, thus during the process we need to maintain a greedy partition whose size is not decreasing either.

\begin{figure}[t]
    \centering
        \includegraphics[scale=0.8]{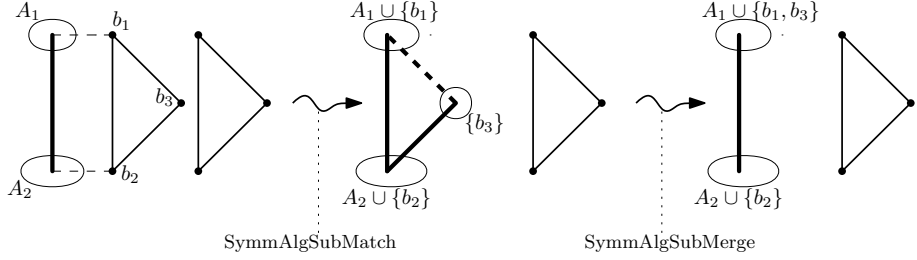}
   \caption{One step of the symmetrization algorithm SymmAlg (dashed lines denote non-edges).}
   \label{fig:alg}
\end{figure}

Now we give the details of the symmetrization. The algorithm SymmAlg  applies the symmetrization algorithms SymmAlgSubMatch and SymmAlgSubMerge alternately, for an example see Figure \ref{fig:alg}.

{\bf SymmAlg:}

We start the process with the given $G_0$ and $P_0$. $P_0$ is a normal greedy partition which can be regarded also as a ggp in which in the first blown-up clique $C_0$ all parts have size $1$.

In a general step of SymmAlg before running SymmAlgSubMatch we have a $G$ and a ggp $P$ of $G$ such that $f(G_0,P_0)\le f(G,P)$. This trivially holds (with equality) before the first run of SymmAlgSubMatch.

{\bf SymmAlgSubMatch:}

If the actual ggp $P$ contains only one part $C_0$ (which is a blow-up of a clique) then we {\bf STOP} SymmAlg.

Otherwise we do the following.
Let the blown-up clique $C_0$ be complete $l$-partite. Temporarily contract the parts of $C_0$ to get a smaller graph in which $P$ becomes a normal greedy partition $P_{temp}$, let $A$ ($|A|=a$) be the first clique (the contraction of $C_0$) and $B=C_1$ ($a\le b=|B|$) be the second clique of $P_{temp}$. As $P$ is a greedy partition, $A\cup B$ must be $K_{b+1}$-free, so we can apply Lemma \ref{lem:twocliques} on $A$ and $B$ to conclude that there is a matching of non-edges between $A$ and $B$ that covers $A$. In $G$ this gives a matching between the parts of the blown-up clique $C_0$ and the vertices of the clique $C_1$ such that if a part $A_i\subset C_0$ is matched with $b_i\in C_1$ then there are no edges in $G$ between $A_i$ and $b_i$.

For every such pair $(A_i,b_i)$ we do the following symmetrization. Let $v\in A_i$ an arbitrary representative of $A_i$ and $w=b_i$. Fix $r_0=r(P)$ and let $f_v=r_0d_v-t_v$ where $d_v$ is the degree of $v$ in $G$ and $t_v$ is the number of triangles in $G$ incident to $v$, or equivalently the number of edges spanned by $N(v)$. Similarly $f_w=r_0d_w-t_w$. Clearly, $f(G,P)=r_0(e-n^2/4)-t=|A_i|f_v+f_w+f_0$ where $f_0$ depends only on the graph induced by the vertices of $V(G)\setminus (A_i\cup\{w\})$. Here we used that there are no edges between $A_i$ and $b_i$. If $f_v\ge f_w$ then we replace $w$ by a copy of $v$ to get the new graph $G_1$, otherwise we replace $A_i$ by $|A_i|$ copies of $w$ to get the new graph $G_1$. In both cases
$$r_0(e_1-n^2/4)-t_1=(|A_i|+1)\max(f_{v},f_{w})+f_0\ge$$$$\ge |A_i|f_v+f_w+f_0=r_0(e-n^2/4)-t.$$
Note that after this symmetrization $V(G)\setminus (A_i\cup\{w\})$ spans the same graph, thus we can do this symmetrization for all pairs $(A_i,b_i)$ one-by-one (during these steps for some vertex $v$ we define $f_v$ using the $d_v$ and $t_v$ of the current graph, while $r_0$ remains fixed) to get the graphs $G_2,G_3,\dots$. At the end we get a graph $G'$ for which
 $$r_0(e'-n^2/4)-t'\ge r_0(e-n^2/4)-t=f(G,P).$$ Now we proceed with SymmAlgSubMerge, which modifies $G'$ further so that the final graph has a ggp of size at least $r_0$.

{\bf SymmAlgSubMerge:}

In this graph $G'$ for all $i$ all vertices in $A_i\cup \{b_i\}$ have the same neighborhood (and form independent sets). Together with the non-matched vertices of $C_1$ regarded as size-$1$ parts we get that in $G'$ the graph induced by $C_0\cup C_1$ is a blow-up of a (not necessarily complete) graph on $b$ vertices. To make this complete we make another series of symmetrization steps. Take an arbitrary pair of parts $V_1$ and $V_2$ which are not connected (together they span an independent set) and symmetrize them as well: take the representatives $v_1\in V_1$ and $v_2\in V_2$ and then $r_0(e'-n^2/4)-t'=|V_1|f_{v_1}+|V_2|f_{v_2}+f_1$ as before, $f_1$ depending only on the subgraph spanned by $G'\setminus (V_1\cup V_2)$. Again replace the vertices of $V_1$ by copies of $v_2$ if $f_2\ge f_1$ and replace the vertices of $V_2$ by copies of $v_1$ otherwise. In the new graph $G'_1$, we have
 $$r_0(e_1'-n^2/4)-t_1'=(|V_i|+|V_j|)\max(f_{v_1},f_{v_2})+f_0\ge$$$$\ge |V_1|f_{v_1}+|V_2|f_{v_2}+f_0=r_0(e'-n^2/4)-t'.$$
 Now $V_1\cup V_2$ becomes one part and in $G'_1$ $C_0\cup C_1$ spans a blow-up $C_0'$ of a (not necessarily complete) graph with $b-1$ parts. Repeating this process we end up with a graph $G''$ for which
$$r_0(e''-n^2/4)-t''\ge r_0(e'-n^2/4)-t'\ge f(G,P).$$
 In $G''$ $C_0\cup C_1$ spans a blow-up $C_0''$ of a complete graph with at most $|C_1|$ parts. Moreover $V\setminus(C_0\cup C_1)$ spans the same graph in $G''$ as in $G$, thus $C_0''$ together with the cliques of $P$ except $C_0$ and $C_1$ have all the requirements to form a ggp $P''$. If the biggest part of $C_0$ was of size $c_l$ then in $C_0'$ this part became one bigger and then it may have been symmetrized during the steps to get $G''$, but in any case the biggest part of $C_0''$ is at least $c_l+1$ big. Thus the size of the new ggp $P''$ is $r(P'')\ge c_l+1+(r(P)-c_l-1)\ge r(P)=r_0$.

If $e''-n^2/4< 0$, then we {\bf STOP} SymmAlg and conclude that we have $f(G_0,P_0)\le f(G,P)\le 0$, finishing the proof.
Otherwise $$f(G'',P'')=r(P'')(e''-n^2/4)-t''\ge r_0(e''-n^2/4)-t''\ge f(G,P)\ge f(G_0,P_0),$$
and so $G'',P''$ is a proper input to SymmAlgSubMatch. We set $G:=G''$ and $P:=P''$ and {\bf GOTO} SymmAlgSubMatch. Note that the number of parts in $P''$ is one less than it was in $P$.
This ends the description of the running of  SymmAlg.

As after each SymmAlgSubMerge the number of cliques in the gpp strictly decreases, SymmAlg must stop until finite many steps.
When SymmAlg STOPs we either can conclude that $f(G_0,P_0)\le 0$ or SymmAlg STOPped because in the current graph $G_*$ the current gpp $P_*$ had only one blow-up of a clique. That is, the final graph $G_*$ is a complete $l_*$-partite graph for some $l_*$ (which has essentially one possible greedy partition). We remark that if the original $G$ was $K_m$-free for some $m$ then $G_*$ is also $K_m$-free, i.e., $l_*\le m-1$. As $f$ never decreased during the process we get using Corollary \ref{cor:onepartggp} that $f(G_0,P_0)\le f(G_*,P_*)\le g(G_*,P_*)\le 0$, finishing the proof of the theorem.
\end{proof}

\section{Remarks}

In the proof of Theorem \ref{thm:tbound}, we can change $f$ to any function that depends on $r,n,e,t,k_4,k_5,\dots $, (where $k_i(G)$ is the number of complete $i$-partite graphs of $G$) and is monotone in $r$ and is linear in the rest of the variables (when $r$ is regarded as a constant) to conclude that the maximum of such an $f$ is reached for some complete multipartite graph. Moreover, as the symmetrization steps do not increase the clique-number of $G$, if the clique number of $G$ is $m$ then this implies that $f(G,P)$ is upper bounded by the maximum of $f(G_*)$ taken on the family of graphs $G_*$ that are complete $m$-partite (some parts can be empty).

Strengthening Theorem \ref{thm:tbound}, it is possible that we can change $f$ to $g$ and the following is also true:

\begin{conj}\label{conj:nice}
if $G$ is a $K_4$-free graph and $r=r(P)$ is the size of an arbitrary greedy partition of $G$ then $t\ge r(e-r(n-r))$ and so $t_e\ge e-r(n-r)$.
\end{conj}

This inequality is nicer than Theorem \ref{thm:tbound} as it holds with equality for all complete $3$-partite graphs. However, we cannot prove it using the same methods, as it is not monotone in $r$. Note that the optimal general bound for $t$ (depending on $e$ and $n$; see \cite{fishersolow} for $K_4$-free graphs and \cite{fisherpaper, razborov} for arbitrary graphs) does not hold with equality for certain complete $3$-partite graphs, thus in a sense this statement would be an improvement on these results for the case of $K_4$-free graphs (by adding a dependence on $r$). More specifically, it is easy to check that there are two different complete $3$-partite graphs with a given $e,n$ (assuming that the required size of the parts is integer), for one of them Fisher's bound holds with equality, but for the other one it does not (while of course Conjecture \ref{conj:nice} holds with equality in both cases).

As we mentioned in the Introduction, in the examples showing that our theorem is sharp, $k$ is roughly at most $n^2/16$ while in general in a $K_4$-free graph $k\le n^2/12$, thus for bigger $k$ it's possible that one can prove a stronger result. Nevertheless, the conjectured bound $t_e\ge e-r(n-r)$ is exact for every $e$ and $r$ as shown by graphs that we get by taking a complete bipartite graph on $r$ and $n-r$ vertices and putting any triangle-free graph in the $n-r$ sized side. For a greedy partition of size $r$ we have $e\le r(n-r)+(n-r)^2/4$ (follows directly from Claim \ref{claim:r2}, see below), thus these examples cover all combinations of $e$ and $r$, except when $e<r(n-r)$ in which case trivially we have at least $0$ triangles, while the lower bound $e-r(n-r)$ on the triangles is smaller than $0$.

\begin{claim}\label{claim:r2}
If $G$ is a $K_4$-free graph, $P$ is a greedy partition of $G$, $r=r(P)$ is the size of $P$ and $r_2$ is the number of cliques in $P$ of size at least $2$, then $e\le r(n-r)+r_2(n-r-r_2)$.
\end{claim}
\begin{proof}
Let $s_1,s_2,s_3$ be the number of size-$1,2,3$ (respectively) cliques of $P$. Then $r=s_1+s_2+s_3,n-r=s_2+2s_3,r_2=s_2+s_3,n-r-r_2=s_3$.
Applying Lemma \ref{lem:twocliques} for every pair of cliques in $P$ we get that the number of edges in $G$ is $e\le {s_1\choose 2}(1\cdot 1-1)+s_1s_2(1\cdot 2-1)+s_1s_3(1\cdot 3-1)+{s_2\choose 2}(2\cdot 2-2)+s_2s_3(2\cdot 3-2)+{s_3\choose 2}(3\cdot 3-3)+s_2+3s_3=s_1s_2+2s_1s_3+s_2^2+4s_2s_3+3s_3^2=(s_1+s_2+s_3)(s_2+2s_3)+(s_2+s_3)s_3=r(n-r)+r_2(n-r-r_2)$.
\end{proof}

Finally, as an additional motivation for Conjecture \ref{conj:nice} we show that Conjecture \ref{conj:nice} holds in the very special case when $G$ is triangle-free, that is $t=t_e=0$. Note that for a triangle-free graph the size-2 cliques of a greedy partition define a non-augmentable matching of $G$.

\begin{claim}
If $G$ is a triangle-free graph and $r=r(P)$ is the size of an arbitrary greedy partition of $G$, i.e., $G$ has a non-augmentable matching on $n-r$ edges, then $0\ge e-r(n-r)$.
\end{claim}

\begin{proof}
We need to show that $e\le r(n-r)$. By Claim \ref{claim:r2}, $e\le r(n-r)+r_2(n-r-r_2)$ where $r_2$ is the number of cliques in $P$ of size at least $2$. If $G$ is triangle-free, then $r_2=n-r$ and so $e\le r(n-r)$ follows.

Let us give another simple proof by induction. As $G$ is triangle-free, $P$ is a partition of $V(G)$ to sets inducing points and edges, thus $ r\le n$.
We proceed by induction on $n-r$. If $n-r=0$ then $P$ is a partition only to points. As $P$ is greedy, $G$ contains no edges, $e=0$ and we are done. In the inductive step, for some $n-r>0$ take a part of $P$ inducing an edge and delete these two points. Now we have a triangle-free graph $G'$ on $n-2$ points and a greedy partition $P'$ of $G'$ that has $r-1$ cliques, thus we can apply induction on $G'$ (as $n'-r'=n-2-(r-1)=n-r-1<n-r$) to conlcude that $G'$ has at most $(r-1)(n-1-r)$ edges. We deleted at most $n-1$ edges, indeed as the graph is triangle-free the deleted two vertices did not have common neighbors, so altogether they had edges to at most $n-2$ other points plus the edge between them. Thus in $G$ we had at most $n-1+(r-1)(n-1-r)=r(n-r)$ edges, finishing the inductive step.
\end{proof}

%

\begin{thebibliography}{5}

\bibitem{bollobas}
B. Bollob\'as, {\sl Extremal graph theory}, Dover Publications, 2004.

\bibitem{chung}
F. R. K. Chung,
On the decompositions of graphs,
{\sl SIAM J. Algebraic and Discrete Methods}, {\bf 2} (1981) 1--12.

\bibitem{ervinsurvey}
E. Gy\H ori,
Edge Disjoint Cliques in Graphs,
{\sl Sets, Graphs and Numbers, Colloquia Mathematica Societatis J\'anos Bolyai}, {\bf 60.} (1991)

\bibitem{ek}
E. Gy\H ori, A. V. Kostochka,
On a problem of G. O. H. Katona and T. Tarj\'an,
{\sl Acta Math. Acad. Sci. Hungar.}, {\bf 34} (1979), 321--327.

\bibitem{fisherpaper}
D. C. Fisher, Lower bounds on the number of triangles in a graph, {\sl J. Graph Theory} {\bf 13} (1989), 505--512.

\bibitem{fishersolow}
D. C. Fisher and A. E. Solow,
Dependence polynomials, {\sl Discrete Mathematics} {\bf 82(3)} (1990), 251--258.

\bibitem{chinese}
Sh. Huang, L. Shi,
Packing Triangles in $K_4$-Free Graphs, {\sl Graphs and Combinatorics}, {\bf 30(3)} (2014), 627-632.

\bibitem{razborov}
A.A. Razborov,
On the minimal density of triangles in graphs, {\sl Comb. Probab. Comput.} {\bf 17} (2008), 603--618.

\end{thebibliography}
\end{document}